\documentclass[11pt,reqno]{amsart} 

\usepackage{amssymb,latexsym}
\usepackage{cite} 

\usepackage[height=190mm,width=130mm]{geometry} 

\theoremstyle{plain}
\newtheorem{theorem}{Theorem}[section]

\newtheorem{corollary}{Corollary}[section]
\newtheorem{proposition}{Proposition}[section]

\theoremstyle{definition}

\theoremstyle{remark}

\numberwithin{equation}{section} 

\begin{document}
\title[The weighted mean matrix with $w_n=2n+1$]{The weighted mean matrix with \\ weight sequence $w_n=2n+1$ is a \\ hyponormal operator on $\ell^2$} 

\author{H. C. Rhaly Jr.}
\address{1081 Buckley Drive\\ Jackson, MS  39206\\ U.S.A.}

\email{rhaly@alumni.virginia.edu}

\begin{abstract}
A weighted mean matrix whose weight sequence is linear with positive coefficients is shown to be a posinormal operator on $\ell^2$.  This operator is also shown to be coposinormal, so it and its adjoint have the same null space and the same range.  The posinormality result leads to a proof that the weighted mean matrix generated by the sequence of odd positive integers is hyponormal and a conjecture concerning a more general case. 
\end{abstract}

\dedicatory{In memory of Russell Aubrey Stokes (1922-2014)}



\maketitle

\section{Introduction}

If $\mathcal{B}(H)$ is the set of all bounded linear operators on a Hilbert space $H$, then the operator $A \in \mathcal{B}(H)$ is \textit{hyponormal}  if \begin{equation*} < (A^*A - AA^*) f , f > \hspace{2mm}  \geq \hspace{2mm} 0 \end{equation*}  for all $f \in H$, and $A $ is said to be  \textit{posinormal} if \begin{equation*} AA^* = A^*PA \end{equation*}  for some positive operator $P \in \mathcal{B}(H)$, called the \textit{interrupter}; $A$ is \textit{coposinormal} if $A^*$ is posinormal.  The following proposition (see \cite{KubrDugg}) contains some key facts about posinormal operators.

\begin{proposition}  \label{propposi} If $A \in \mathcal{B}(H)$, then the following are equivalent.
\begin{enumerate}
 \item There exists a nonnegative $P \in \mathcal{B}(H)$ such that $AA^* = A^*PA$ (i.e., $A$ is posinormal).
 \item There exists a nonnegative $P \in \mathcal{B}(H)$ such that $AA^* \leq A^*PA$.
 \item There exists a nonnegative $\alpha \in \mathbb{R}$ such that $AA^* \leq \alpha^2 A^*A$.
 \item $Ran(A) \subseteq Ran(A^*)$, where $Ran(A) = \{g \in H : g = Af$ for some $f \in H\}$.
 \item  There exists a $B \in \mathcal{B}(H)$ such that $A = A^*B$.
 \vspace{1mm}
 
 \noindent Moreover, each of the above assertions implies the following one.
 
 \vspace{1mm}
 
 \item $Ker (A) \subseteq Ker (A^*)$, where $Ker (A) = \{f \in H: Af = 0 \}$.
 
 \vspace{1mm}
 
 \noindent Furthermore, if $Ran (A)$ is closed, then these six assertions are all \\ equivalent.

\end{enumerate}

\end{proposition}

\begin{proof}  See \cite{KubrDugg}.
\end{proof}

\noindent Note that $A$ is hyponormal if part (c) of the proposition is satisfied for $\alpha = 1$, so hyponormal operators are necessarily posinormal.

A lower triangular infinite matrix $M =  [m_{ij}] \in \mathcal{B}(\ell^2)$ is {\it factorable} if its entries are of the form

\begin{equation*} m_{ij} = \left \{ \begin{array}{lll}
a_ic_j & if  &  j \leq i\\
0 & if & j > i \end{array}\right. \end{equation*}

\noindent where $a_i$ depends only on $i$ and $c_j$ depends only on $j$.   A \textit{weighted mean matrix} is a lower triangular matrix with entries $w_j/W_i$, where $\{w_j\}$ is a nonnegative sequence with $w_0 > 0$, and $W_i = \sum_{j=0}^i w_j$.   A weighted mean matrix is factorable, with $a_i = 1/W_i$  and $c_j = w_j$ for all $i$,$j$.   In recent decades these operators, acting on various sequence spaces, have been studied by many authors in a number of papers, including \cite{BorwGao}, \cite{GaoP01}, \cite{GaoP02}, \cite{Harr}, \cite{PazYous}, \cite{Rho01},and \cite{Rho02}.   Hyponormality on $\ell^2$ has been demonstrated in three special cases, namely, for the operators generated by $w_j = j+1$ (see \cite{RhaRho02}),  $w_j = (j+1)^2$ (see \cite{RhaRho03}), and $w_j=2-1/2^j$ (see \cite{RhaRho01}); the first and third examples were handled without software assistance, in contrast with what will be the case in this paper.

Initially under consideration here will be the weighted mean matrix $M$ associated with the general linear weight sequence $w_n = an+b$ where $a$, $b > 0$.  The first aim  is to show that this operator is both posinormal and coposinormal.  Posinormality will be useful for achieving the two main objectives of this paper, (1) showing that the matrix associated with the sequence $w_n=2n+1$ is a hyponormal operator on $\ell^2$ and (2) propounding a conjecture for hyponormality of the more general case.

\section{Posinormality and hyponormality}

Our first goal is to demonstrate that a weighted mean matrix $M$ whose weight sequence is linear with positive coefficients is a posinormal operator on $\ell^2$.   The following theorem from \cite{Rhal01} provides a tool for that purpose.

\begin{theorem}  \label{posi} Suppose $M = [a_ic_j]$ is a lower triangular factorable matrix that acts as a bounded operator on $\ell^2$  and that the following conditions are satisfied:
\begin{enumerate}

\item both $\{a_n\}$ and $\{a_n/c_n\}$ are positive decreasing sequences that converge to 0, and 

\item the matrix $B$ defined by $B = [b_{ij}]$ by

\begin{equation*} b_{ij} = \left \{ \begin{array}{lll}
c_i (\frac {1}{c_j} - \frac {1}{c_{j+1}}\frac {a_{j+1}}{a_j}) & if & i \leq j;\\
-\frac {a_{j+1}}{a_j} & if  &  i = j+1;\\
0 & if    & i > j+1. \end{array}\right.  \end{equation*}  is a bounded operator on $\ell^2$.  

\end{enumerate}

 \noindent Then $M$  is posinormal. \end{theorem}

\begin{proof}  See \cite{Rhal01}.
\end{proof}

\begin{theorem}  \label{linearposi} The weighted mean matrix $M$ associated with the weight sequence $w_n = an+b$ with $a$, $b >0$ is a posinormal operator on $\ell^2$.
\end{theorem}

\begin{proof}  It suffices to consider sequences of the form $w_n = kn+1$ with $k>0$.  It is clear that condition (a) of Theorem \ref{posi} is satisfied, so we focus our attention now on condition (b).  Since $c_n=kn+1$ and $a_n= \frac{2}{(n+1)(kn+2)}$, the matrix $B = [b_{ij}]$ entries become
\begin{equation*} b_{ij} = \left \{ \begin{array}{lll}
(ki+1) \cdot \frac{3k^2j^2+(5k^2+6k)j+2k^2+6k+2}{(j+2)(kj+1)(kj+k+1)(kj+k+2)} & if & i \leq j;\\
-\frac {(j+1)(kj+2)}{(j+2)(kj+k+2)} & if  &  i = j+1;\\
0 & if    & i > j+1. \end{array}\right. \end{equation*}

\noindent Note that for $i \leq j$, $0 < b_{ij} \leq 3/(j+2)$ for all $j$, so the upper triangular part $T$ of matrix $B$ is a bounded operator on $\ell^2$.  If $W$ is the weighted shift with weight sequence $\{\frac{(n+1)(kn+2)}{(n+2)(kn+k+2)} : n \geq 0 \}$, then $B = T - W$, so $B \in \mathcal{B}(\ell^2)$.  It now follows from Theorem \ref{posi} that $M$ is posinormal.
\end{proof}

We note that the range of $M$ contains all the $e_n$'s from the standard orthonormal basis for $\ell^2$.  From \cite{Rhal01} we know that the entries of the interrupter $P= [p_{ij}]=B^*B$  are given by

\begin{equation*} p_{ij} = \left \{ \begin{array}{lll}
\frac {c_j^2 c_{j+1}^2 a_{j+1}^2 + (\sum_{\ell=0}^j c_{\ell}^2) (c_{j+1}a_j-c_j a_{j+1})^2}{c_j^2 c_{j+1}^2 a_j^2} & if    & i = j;\\
\frac {(c_i a_{i+1} - c_{i+1} a_i) [c_j (\sum_{\ell=0}^{j+1} c_{\ell}^2)a_{j+1} -  c_{j+1} (\sum_{\ell=0}^j c_{\ell}^2 )a_j]}{c_i c_{i+1} c_j c_{j+1} a_i a_j} & if    & i > j;\\
\frac  {(c_j a_{j+1} - c_{j+1} a_j) [c_i (\sum_{\ell=0}^{i+1} c_{\ell}^2)a_{i+1} -  c_{i+1}( \sum_{\ell=0}^i c_{\ell}^2 )a_i]}{c_i c_{i+1} c_j c_{j+1} a_i a_j} & if  &  i < j. \end{array}\right. \end{equation*} 
For $M$ to be hyponormal, it must be true that \begin{equation*} \langle (M^*M - MM^*) f , f \rangle= \langle (M^*M - M^*PM) f , f \rangle =\langle (I -P) Mf , Mf \rangle  \geq 0 \end{equation*} for all $f$ in $\ell^2$.   Consequently, we can conclude that $M$ will be hyponormal when $Q : \equiv  I - P \ge 0$.   Again using $c_n=kn+1$ and $a_n= \frac{2}{(n+1)(kn+2)}$, we determine that the entries of $Q$ = [$q_{ij}$] are given by $q_{ij} = $  \small  \begin{equation} \label{eq:display}  \left \{ \begin{array}{lll}
\frac{6 k^6j^6 + (21k^6+42k^5) j^5 + (25k^6+126k^5+114k^4) j^4 + (9k^6+126k^5+266k^4+168k^3) j^3}{6 (j+2)(kj+1)^2 (kj+k+1)^2 (kj+k+2)^2} &     & \\ 
 \hspace{5mm} + \frac{(-3k^6+42k^5+190k^4+264k^3+160k^2) j^2 + (-2k^6+38k^4+96k^3+148k^2+96k) j}{6 (j+2)(kj+1)^2 (kj+k+1)^2 (kj+k+2)^2} &     & \\
 \hspace{15mm} + \frac{12k^2+48k+24}{6 (j+2)(kj+1)^2 (kj+k+1)^2 (kj+k+2)^2} & if    & i=j;\\

-\frac{1}{6} \cdot \frac{3 k^2 i^2 + (5k^2+6k) i + (2k^2+6k+2)}{(i+2) (ki+1) (ki+k+1)(ki+k+2)} \cdot \frac{(j+1)(k^4j^2 + (k^4 + 4k^2)j + 6k)}{(kj+1) (kj+k+1)(kj+k+2)} & if & i >  j;\\
-\frac{1}{6} \cdot \frac{3 k^2 j^2 + (5k^2+6k) j + (2k^2+6k+2)}{(j+2) (kj+1) (kj+k+1)(kj+k+2)} \cdot \frac{(i+1)(k^4i^2 + (k^4 + 4k^2)i + 6k)}{(ki+1) (ki+k+1)(ki+k+2)} & if  &  i < j. \end{array}\right. \end{equation} \normalsize

When $k=1$, $Q$ assumes the form studied in \cite{RhaRho02}.  When $k \ne 1$, the key lemma in \cite{RhaRho02} is not satisfied, so a different approach is required here.  Continuing toward hyponormality in the general case has heretofore not been fruitful, so we focus our attention now on the case $k=2$.

\begin{theorem}  \label{odd}  The weighted mean matrix $M$ associated with the weight sequence $w_n = 2n+1$ is a hyponormal operator on $\ell^2$.
\end{theorem} 

\begin{proof}  Using $k=2$, we find that the entries of $Q$ = [$q_{ij}$], in irreducible form, are  
\begin{equation*} q_{ij} = \left \{ \begin{array}{lll}
\frac{12j^4+60j^3+104j^2+66j+7}{3 (j+2)^3 (2j+1)(2j+3)} & if    & i = j;\\
-\frac{1}{3} \cdot \frac{6i^2+16i+11}{(i+2)^2 (2i+1) (2i+3)} \cdot \frac{j+1}{j + 2}& if    & i > j;\\
-\frac{1}{3} \cdot \frac{6j^2+16j+11}{(j+2)^2 (2j+1) (2j+3)} \cdot \frac{i+1}{i + 2} & if  &  i < j. \end{array}\right.  \end{equation*}

\noindent In order to show that $Q$ is positive, it suffices to show that $Q_N$, the $N^{th}$ finite section of $Q$ (involving rows $i=0,1,2,...,N$ and columns $j=0,1,2,...,N$), has positive determinant for each positive integer $N$.   For columns $j=0,1,...,\mbox{N-1}$, we multiply the $(j+1)^{st}$ column of $Q_N$ by $z_j :\equiv \frac{(j+1)(j+3)}{(j+2)^2}$ and subtract from the $j^{th}$ column.  Call the new matrix $Q_N^{'}$.  Then we work with the rows of  $Q_N^{'}$.   For $i=0,1,...,N-1$,  we multiply the $(i+1)^{st}$ row of $Q_N^{'}$ by $z_i$ and subtract from the $i^{th}$ row.  This leads to the tridiagonal form 

\begin{equation*} Y_N :\equiv \left(\begin{array}{cccccc}
d_0 & s_0 & 0 & \hdots & 0 & 0\\
s_0 & d_1 & s_1 & \hdots & 0 & 0\\
0 & s_1 & d_2 & \hdots & . & 0\\
\vdots & \vdots & \vdots & \ddots & \vdots & \vdots\\
0 & 0 & . & \hdots & d_{N-1} & s_{N-1}\\
0 & 0 & 0 & \hdots & s_{N-1} & d_N \end{array}\right), \end{equation*}

\noindent where 
\begin{align*} d_n = q_{nn} - z_nq_{n,n+1} -z_n(q_{n+1,n}-z_nq_{n+1,n+1}) = q_{nn} -2z_nq_{n,n+1} + z^2_nq_{n+1,n+1}   \\
 = \frac{16n^7 +  200 n^6 +1024 n^5 + 2768 n^4 + 4226 n^3 + 3576 n^2 +  1481 n + 197}{(n+2)^4(n+3)(2n+1)(2n+3)(2n+5)}  \end{align*}  and
\begin{equation*} s_n = q_{n+1,n} - z_nq_{n+1,n+1} = - \frac{(n+1)(2n^2+8n+7)}{(n+2)^2(n+3)(2n+3)}  \textrm{ when } 0  \le n \le N-1; \end{equation*} and lastly, 
\begin{equation*} d_N = q_{NN} =  \frac{12N^4+60N^3+104N^2+66n+7 }{3(N+2)^3 (2N+1)(2N+3)}.  \end{equation*}

\noindent Note that $\det Y_N = \det Q_N^{'} = \det Q_N$.  Next we transform $Y_N$ into a triangular matrix with the same determinant.  The new matrix has diagonal entries $\delta_n$ which are given by the recursion formula:  $\delta_0 = d_0$, $\delta_n = d_n- s_{n-1}^2 / \delta_{n-1}  \hspace{1mm} (1 \le n \le N)$.   An induction argument shows that  \begin{equation*} \delta_n > \frac{4n+10}{4n^2+20n+37}  \end{equation*} for $0 \leq n \leq N-1$.    Since $d_N$ departs from the pattern set by the earlier $d_n$'s, $\delta_N$ needs to be handled separately:  \small
\begin{equation*} \delta_N \geq \frac{24N^8+140N^7+432N^6+1160N^5+2234N^4+2297N^3+1070N^2+216N+14 }{6(N+1)^4(N+2)^3(2N+1)^2(2N+3)}.  \end{equation*} \normalsize
\noindent Therefore $\det{Q_N} = \prod_{n=0}^N \delta_n > 0$, so the proof is complete.
\end{proof}

The verification of the induction step above reduces -- using the SageMath software system \cite{sage} -- to the observation that $96 n^{10} + 4944 n^{9} + 54624 n^{8} + 282288 n^{7} + 824294 n^{6} + 1447767 n^{5} + 1531563 n^{4} + 927504 n^{3} + 285132 n^{2} + 35022 n + 178 \geq 0$ for all $n \geq 1$.  The reduction was accomplished when SageMath executed the following command.  \\

\noindent  expand($(16*n$\textasciicircum$7 + 200*n$\textasciicircum$6 + 1024*n$\textasciicircum$5 +2768*n$\textasciicircum$4 + 4226*n$\textasciicircum$3 + 3576*n$\textasciicircum$2+ 1481*n + 197)*2*(n+1)$\textasciicircum$4*(2*n+1)*(4*n$\textasciicircum$2+20*n+37)-n$\textasciicircum$2*(2*n$\textasciicircum$2+4*n+1)$\textasciicircum$2*(4*n$\textasciicircum$2+12*n+21)*(n+2)$\textasciicircum$2*(n+3)*(2*n+5)*(4*n$\textasciicircum$2+20*n+37)-(4*n+10)*2*(n+1)$\textasciicircum$4*(n+2)$\textasciicircum$4*(n+3)*(2*n+1)$\textasciicircum$2*(2*n+3)*(2*n+5))$  \\

\section{Coposinormality}

We  now proceed to show that a weighted mean operator whose weight sequence is linear with positive coefficients has a posinormal adjoint.  The next theorem provides a tool for doing that.

\begin{theorem}   \label{coposi} Suppose $M = [a_ic_j]$ is a lower triangular factorable matrix that acts as a bounded operator on $\ell^2$  and that the following conditions are satisfied:

\begin{enumerate}
\item both $\{a_n\}$ and $\{a_n/c_n\}$ are positive decreasing sequences that converge to 0, and
\item the matrix $Z= [z_{ij}]$ defined by

\begin{equation*}z_{ij} = \left \{ \begin{array}{lll}
\frac {a_i}{a_j}  & if    & j = 0;\\ 
a_i (\frac {1}{a_j} - \frac {c_{j-1}}{c_j}\frac {1}{a_{j-1}}) & if    & 0 < j \leq i;\\
-\frac {c_{j-1}}{c_j} & if  &  j = i+1;\\
0 & if    & j > i+1. \end{array}\right. \end{equation*} 
 is a bounded operator on $\ell^2$,

\end{enumerate}

\noindent
Then $M$* is posinormal.
\end{theorem}

\noindent
\begin{proof}   See \cite{Rhal01}.
\end{proof}

\begin{theorem} \label{linearcoposi} The weighted mean matrix $M$ associated with the weight sequence $w_n = an+b$ with $a$, $b >0$ is a coposinormal operator on $\ell^2$.
\end{theorem}

\begin{proof}   Again it suffices to consider sequences of the form $w_n = kn+1$ with $k>0$.  It is clear that condition (a) of Theorem \ref{coposi} is satisfied, so we focus our attention now on condition (b).  When $c_n=kn+1$ and $a_n=\frac{2}{(n+1)(kn+2)}$, the matrix $Z = [z_{ij}]$ becomes
\begin{equation*} z_{ij} = \left \{ \begin{array}{lll}
\frac {2}{(i+1)(ki+2)}& if    & j = 0;\\
\frac{1}{(i+1)(ki+2)} \cdot \frac{3k^2j^2+(6k-k^2)j+2}{kj+1} & if    & 0 < j \leq i;\\
-\frac {kj-k+1}{kj+1} & if  &  j = i+1;\\
0 & if    & j > i+1. \end{array}\right. \end{equation*}
\noindent Note that for $j \leq i$, $0 < z_{ij} \leq 3/(i+1)$ for all $i$, so the lower triangular part $T$ of matrix $Z$ is a bounded operator on $\ell^2$.  If $W$ is the weighted shift with weight sequence $\{(kn+1)/(kn+k+1) : n \geq 0 \}$, then $Z = T - W^*$, so $Z \in B(\ell^2)$.  It now follows from Theorem \ref{coposi} that $M^*$ is posinormal.
\end{proof}

\begin{corollary}  \label{powers} If $M$ is the weighted mean operator on $\ell^2$ determined by the sequence $w_n = an+b$ with $a$, $b > 0$, then $M^j$ is both posinormal and coposinormal for each positive integer $j$.
\end{corollary}

\begin{proof}  This follows from Theorems \ref{linearposi} and \ref{linearcoposi}, together with \cite[Corollary $1$(b)]{Kubr02}.  \end{proof}

\begin{theorem}  If $M$ is the weighted mean operator on $\ell^2$ determined by the sequence $w_n = an+b$ with $a$, $b > 0$, then both $M$ and $M^*$ are injective and have dense range with \begin{equation*} Ran (M) = Ran (M^*). \end{equation*}
\end{theorem}

\begin{proof}  Since $M$ is both posinormal and coposinormal, it follows from Proposition \ref{propposi} that $Ker (M) = Ker (M^*)$ and $Ran (M) = Ran (M^*)$.  It is easy to see that $Ker (M) = \{0\}$.  Consequently, both $M$ and $M^*$ are one-to-one, and both have dense range. 
\end{proof}

 \section{Concluding Remarks}

After seeing the hyponormality result for $w_n=2n+1$, one might reasonably wonder whether the calculations for the weighted mean matrix for weight sequence $w_n=3n+1$ are significantly more complex, and the answer is -- Yes.  To illustrate, we observe that for $w_n = 2n+1$, the irreducible form of $q_{nn}$ has a degree $4$ numerator and a degree $5$ denominator (see the proof of Theorem \ref{odd}).  In contrast, for $w_n = 3n+1$, the irreducible form of $q_{nn}$ has a degree $6$ numerator and a degree $7$ denominator -- as well as larger coefficients for all the terms, so the hyponormality calculations for that weight sequence are significantly more complicated than what was presented here.  Nevertheless, we offer the following for the reader to consider:

\begin{itemize}  \item Successful calculations have been run for $w_n=kn+1$ when  $k=4/3$, $3/2$, $3$, and $4$, besides $k=2$ (for $k=1$, see \cite{RhaRho02}).  \item The first four finite sections of $Q$, whose general entries are displayed in \ref{eq:display}, have positive determinants for all $k>0$ (see Appendix). \end{itemize}   Consequently, we close with the following conjecture.

\vspace{2mm}

{\bf Conjecture 4.1.} {\it If $M$ is a weighted mean matrix whose weight sequence is $w_n = an + b$ with $a \geq b >0$, then $M$ is a hyponormal operator on $\ell^2$.}

\section{Appendix}

The expressions below refer to the second bulleted item immediately preceding Conjecture $4.1$.  The determinant of the first finite section of $Q$ is \begin{equation*} \det(Q_0) = \frac{k^2+4k+2}{(k+1)^2(k+2)^2}, \end{equation*} the determinant of the second is  \begin{equation*} \det(Q_1) = \frac{7k^6+56k^5+114k^4+110k^3+76k^2+36k+6}{9(k+1)^4(k+2)^2(2k+1)^2}, \end{equation*}   the determinant of the third is \begin{align*} \det(Q_2) = \frac{405k^{10}+4320k^9+13959k^8+22968k^7+24214k^6+18464k^5}{36(k+1)^4(k+2)^2(2k+1)^2(3k+1)^2(3k+2)^2} \\
                      +  \frac{10527k^4+4496k^3+1423k^2+288k+24}{36(k+1)^4(k+2)^2(2k+1)^2(3k+1)^2(3k+2)^2 },  \end{align*} 
\noindent and the determinant of the fourth is  \begin{align*} \det(Q_3) = \frac{101088k^{14}+1280448k^{13}+5308128k^{12}+11654028k^{11}}{900(k+1)^4(k+2)^2(2k+1)^4(3k+1)^2(3k+2)^2(4k+1)^2} \\ 
+  \frac{16473861k^{10}+16665540k^9+12793267 k^8+7701000k^7}{900(k+1)^4(k+2)^2(2k+1)^4(3k+1)^2(3k+2)^2(4k+1)^2}  \\  +  \frac{3719530k^6+1460772k^5+463575k^4+115268k^3}{900(k+1)^4(k+2)^2(2k+1)^4(3k+1)^2(3k+2)^2(4k+1)^2} \\  +
\frac{20975 k^2+2400k+120}{900(k+1)^4(k+2)^2(2k+1)^4(3k+1)^2(3k+2)^2(4k+1)^2} . \end{align*} 
  \noindent These calculations were assisted by \cite{sage}.  Clearly each of these determinants is positive for all $k>0$.

\end{document}